\documentclass[10pt]{amsart}

\usepackage{amssymb}
\usepackage{amsmath}
\usepackage{epsfig}

\usepackage{graphics}

\newcommand{\vvert}{|\!|\!|}
\newcommand{\RR}{{\mathbb R}}
\newcommand{\CC}{{\mathbb C}}
\newcommand{\NN}{{\mathbb N}}
\newcommand{\ZZ}{\mathbb Z}

\newcommand{\EE}{{\mathbb E}}

\def\d{{\rm d}}

\def\B{{\mathcal B}}

\def\E{{\mathcal E}}

\def\P{{\mathcal P}}

\def\S{{\mathcal S}}
\def\T{{\mathcal T}}

\numberwithin{equation}{section}

\newtheorem{theo}{Theorem}
\newtheorem{prop}[theo]{Proposition}
\newtheorem{coro}[theo]{Corollary}
\newtheorem{lemma}[theo]{Lemma}
\newtheorem{defi}[theo]{Definition}

\theoremstyle{remark}
\newtheorem{remark}[theo]{Remark}

\begin{document}

\parindent = 0cm

\title
{
Necessary and sufficient conditions to be an eigenvalue for linearly recurrent
dynamical Cantor systems
}

\author{Xavier Bressaud}
\address{Centro de Modelamiento Ma\-te\-m\'a\-ti\-co
UMR 2071 UCHILE-CNRS, Casilla 170/3 correo 3,
Santiago, Chile, and
Institut de Math\'ematiques de Luminy,
163 avenue de Luminy, Case 907, 13288 Marseille Cedex 9, France.}
\email{bressaud@dim.uchile.cl, bressaud@iml.univ-mrs.fr}

\author{Fabien Durand}
\address{Laboratoire Ami\'enois
de Math\'ematiques Fondamentales et
Appliqu\'ees, CNRS-UMR 6140, Universit\'{e} de Picardie
Jules Verne, 33 rue Saint Leu, 80000 Amiens, France.}
\email{fabien.durand@u-picardie.fr}

\author{Alejandro Maass}
\address{Departamento de Ingenier\'{\i}a
Matem\'atica, Universidad de Chile
and Centro de Modelamiento Ma\-te\-m\'a\-ti\-co,
UMR 2071 UCHILE-CNRS, Casilla 170/3 correo 3,
Santiago, Chile.}
\email{amaass@dim.uchile.cl}


\begin{abstract}
We give necessary and sufficient conditions to have measurable and
continuous eigenfunctions for linearly recurrent
Cantor dynamical systems.
We also construct explicitly an example of linearly recurrent system
with  nontrivial Kronecker factor and a trivial maximal equicontinuous factor.
\end{abstract}

\date{July 4, 2004}
\maketitle
\markboth{Xavier Bressaud, Fabien Durand, Alejandro
Maass}{Necessary and sufficient conditions to be an eigenvalue }

\section{Introduction}

Let $(X,T)$ be a topological dynamical system, that is,
$X$ is a compact metric space and $T:X \to X$ is a homeomorphism.
Let $\mu$ be a $T$-invariant probability measure on $X$.
In the classification of dynamical systems in ergodic theory and
topological dynamics rotation factors play a central role. In the
measure theoretical context this is reflected by
the existence of a $T$-invariant sub $\sigma$-algebra
$\mathcal{K}_\mu$ of the Borel $\sigma$-algebra of $X$, $\mathcal{B}_X$,
such that $$L^2(X,\mathcal{K}_\mu,\mu)
= \overline{<\{ f\in L^2(X,\B_X,\mu)\setminus \{0\};
\exists \lambda \in \CC,
f\circ T= \lambda f \}>}.$$
It is the subspace spanned by the eigenfunctions which determines
the Kronecker factor.
From a purely topological point of view the role of the Kronecker factor
is played by the maximal equicontinuous factor.
It can be defined in several ways. When $(X,T)$ is minimal
(all orbits are dense), it is determined by the continuous eigenfunctions.
So it is relevant to ask
whether there exist continuous eigenfunctions; or even under which
conditions measure theoretical eigenvalues can be
associated to continuous eigenfunctions.

In \cite{CDHM} these questions are considered for {\it linearly
recurrent systems}. These systems are characterized by the
existence of a nested sequence of clopen (for closed and
open) Kakutani-Rohlin (CKR) partitions of the system $(\P(n);
n\in \NN)$ verifying some technical conditions we call {\bf (KR1)},
{\bf (KR2)},..., {\bf (KR6)} (see below),  and
such that the height of the towers of each partition
increases ``linearly'' from one level to the other. A partial answer
to the former question is given in terms of the sequence of
matrices $(M(n); n \geq 1)$ relating towers from different levels
in \cite{CDHM}.
A complete answer to this question is given in the next theorem.

We need some extra notations.
For each real number $x$ we write $\vvert x\vvert$ for the distance of $x$
to the nearest integer.
For a vector $V=(v_1,\dots,v_m)^T \in \RR^m$, we write
$$
\Vert V\Vert=\max_{1\leq j\leq m}|v_j|\text{ and }
\vvert V\vvert=\max_{1\leq j\leq m}\vvert v_j\vvert\ .
$$
For $n\geq 2$ we put $P(n)=M(n) \cdots M(2)$ and
$H(1)=M(1)$.

\begin{theo}
\label{ssi}
Let $(X,T)$ be a linearly recurrent Cantor system given
by an increasing sequence of CKR
partitions with associated matrices
$(M(n); n \geq 1)$, and let $\mu$ be the unique invariant measure.
Let $\lambda=\exp(2i\pi \alpha)$.
\begin{enumerate}
\item
\label{measurable}
$\lambda$ is an eigenvalue of $(X,T)$ with respect to $\mu$ if and only if
$$\displaystyle \sum_{n\ge 2} \vvert \alpha P(n) H(1) \vvert^2
<\infty \ . $$
\item
\label{continuous}
$\lambda$ is a continuous eigenvalue of
$(X,T)$ if and only if
$$\displaystyle \sum_{n\ge 2}\vvert \alpha P(n) H(1)
\vvert <\infty \ . $$
\end{enumerate}
\end{theo}

In \cite{CDHM} the authors prove the necessary condition in the statement
\eqref{measurable} and the sufficient condition in the statement \eqref{continuous}.
One of the most relevant facts is that both
conditions do not depend on
the order of levels in the towers defining the system but just
on the matrices.

\section{Definitions and background}
\label{definitions}

\subsection{Dynamical systems}

By a {\it topological dynamical system} we mean a couple $(X,T)$
where $X$ is a compact metric space and $T: X \to X$ is a homeomorphism.
We say that it is a {\it Cantor system} if $ X $ is a Cantor space;
that is, $ X $
has a countable basis of its topology which consists of closed and
open sets
(clopen sets) and does not have isolated points.
We only deal here with minimal Cantor systems.

A complex number $\lambda$ is a {\it continuous eigenvalue}
of $(X,T)$ if there exists a continuous function $f : X\to \CC$,
$f\not = 0$, such that
$f\circ T = \lambda f$; $f$ is called a {\it continuous
eigenfunction}
(associated to $\lambda$).
Let $\mu$ be a
$T$-invariant probability measure, i.e., $T\mu = \mu$, defined on
the Borel $\sigma$-algebra $\B_X$ of $X$. A complex number
$\lambda$ is an {\it eigenvalue} of the dynamical system $(X,T)$
with respect to $\mu$ if there exists $f\in L^2 (X,\B_X,\mu)$,
$f\not = 0$, such that $f\circ T = \lambda f$; $f$ is called an
{\it eigenfunction} (associated to $\lambda$). If the system is
ergodic, then every eigenvalue is of modulus 1, and every
eigenfunction has a constant modulus. Of course continuous eigenvalues are
eigenvalues.

In this paper we mainly consider topological dynamical systems
$(X,T)$ which are uniquely ergodic and minimal.
That is, systems that admit a unique
invariant probability measure which is ergodic, and such that
the unique $T$-invariant sets are $X$ and $\emptyset$.

\subsection{Partitions and towers}

\label{subsecpart}

Sequences of partitions associated to minimal Cantor
systems were used in \cite{HPS} to build representations
of such systems as adic transformations on ordered
Bratteli diagrams. Here we do not introduce the whole formalism of
Bratteli diagrams since we will only use the language describing
the tower structure. Both languages are very close.
We recall some definitions and fix some notations.

Let $(X,T)$ be a minimal Cantor system.
A {\it clopen Kakutani-Rokhlin partition} (CKR partition) is a
partition $\P$ of $X$ given by
\begin{equation}
\label{eq:def-KR}
\P = \{ T^{-j} B_k ; 1\leq k\leq C , \ 0 \leq j< h_k \}
\end{equation}
where $C$ is a positive integer,
$B_1,\dots , B_C$ are clopen subsets of $X$ and
$h_1,\dots,h_k$ are positive integers.
For $1\leq k\leq C$, the $k$-th {\it tower} of $\P$ is

$$
\T_k=\bigcup_{j=0}^{h_k-1} T^{-j} B_k
$$

and its height is $h_k$;
the {\it roof} of $\P$ is the set $B =\bigcup_{1\leq k\leq C}B_k$.
Let
\begin{equation}
\label{eq:def-seq-KR}
\bigl(
\P (n)=
\{
T^{-j}B_{k} (n); 1\leq k\leq C(n),\ 0\le j<h_{k}(n)
\}
\ ; \ n\in\NN
\bigr)
\end{equation}
be a sequence of CKR partitions. For every $n\in \NN$ and $1\leq k\leq
C(n)$,
$B(n)$ is the roof of $\P (n)$ and $\T_k(n)$ is the
$k$-th tower of $\P(n)$. We assume that $\P (0)$ is the
trivial partition, that is, $B (0)=X$, $C(0)=1$ and $h_{1} (0) = 1$.

We say that $(\P (n);n\in \NN)$ is \emph{nested} if for every
$n\in \NN$ it satisfies:

\medskip

{\bf (KR1)} $B (n+1) \subseteq B (n)$;

\medskip

{\bf (KR2)} $\P (n+1) \succeq \P (n)$; i.e., for all $A\in \P (n+1)$
there exists $A^{'}\in \P (n)$ such that $A\subseteq A^{'}$;

\medskip

{\bf (KR3)} $ \bigcap_{n\in \NN} B (n) $ consists of a unique point;

\medskip

{\bf (KR4)} the sequence of partitions spans the topology of $X$;

\medskip

In \cite{HPS} it is proven that given a minimal Cantor system
$(X,T)$ there exists a nested sequence of CKR partitions fulfilling
{\bf (KR1)}--{\bf (KR4)} ({\bf (KR1)}, {\bf (KR2)}, {\bf (KR3)} and
{\bf (KR4)}) and the following
additional technical conditions:

\medskip

{\bf (KR5)} for all $n\geq 1$, $1\leq k \leq C (n-1)$, $1\leq l\leq C
(n)$,
there exists $0 \leq j < h_{l} (n)$ such that $T^{-j} B_{l}(n)
\subseteq B_{k}(n-1)$;

\medskip

{\bf (KR6)} for all $n\geq 1$, $B (n) \subseteq B_{1}(n-1)$.

\medskip

We associate to $(\P (n); n \in \NN)$
the sequence of matrices $(M(n) ; n\ge 1 )$, where
$M (n) = (m_{l,k}(n); 1\leq l \leq C (n) , 1\leq k \leq C (n-1))$ is
given by
$$
m_{l,k}(n) =
\#
\{
0 \leq j < h_{l}(n) ; T^{-j} B_{l}(n) \subseteq B_{k}(n-1)
\}.
$$
Notice that {\bf (KR5)} is equivalent to:
for all $n\ge 1$, $M(n)$ has strictly positive entries.
For $n\geq 0$ set $H(n)=(h_{l}(n) ; 1\leq l\leq C (n))^T$.
As the sequence of partitions is nested
$H (n)=M (n)H (n-1)$ for $n\geq 1$. Notice that $H(1)=M(1)$.
For $n>m\geq 0$ we define
$$
P(n,m)=M(n)M(n-1)\dots M(m+1) \hbox{ and } P(n)=P(n,1)\ .
$$
Clearly
$$
P_{l,k}(n,m)=\#\bigl\{ 0\leq j < h_l(n);\
T^{-j}B_l(n)\subseteq B_k(m) \bigr \},
$$
for $1 \leq l \leq C(n), \ 1 \leq k \leq C(m)$, and
$$
P(n,m)H(m)= H(n)=P(n)H(1)\ .
$$

\begin{figure}
\vspace{0.375cm}
\resizebox*{0.7\textwidth}{!}{\input{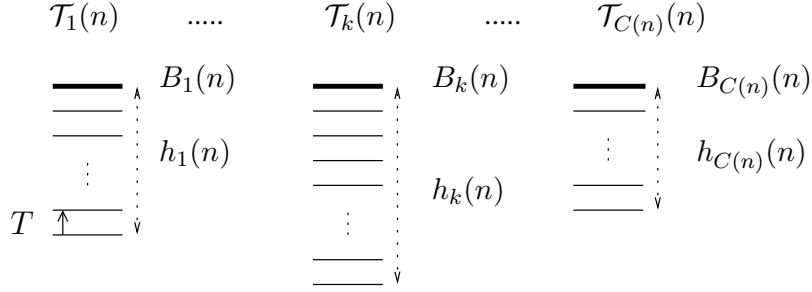}}
\vspace{0.375cm} \caption{CKR partition of level $n$: (a) $X$ is
partitioned in $C(n)$ towers. Each tower $\mathcal{T}_k(n)$, $1
\leq k \leq C(n)$, is composed by $h_k (n)$ disjoint sets, called
{\it stages} of the tower. The top of a tower is the roof $B_k(n)$. (b)
The dynamics of $T$ consists in going up from one stage to the other
of a tower until the roof. Points in a roof are sent to the
bottom of the towers; two points in the same roof can be send to
different towers. \label{Figure1}}
\end{figure}

\subsection{Linearly recurrent systems}
\label{linrec}

The notion of linearly recurrent minimal Cantor system
(also called linearly recurrent system) in the generality
we present below was stated in \cite{CDHM}. It is an extension of the
concept of linearly recurrent subshift introduced in \cite{DHS}.

\begin{defi}
A minimal Cantor system $(X,T)$ is {\it linearly recurrent} (with
constant $L$) if there exists a nested sequence of
CKR partitions
$(\P(n)=\{T^{-j}B_{k}(n); 1\leq k\leq C(n), 0\le j<h_{k}(n)\};n\in\NN
)$
satisfying {\bf (KR1)}--{\bf (KR6)} and

\medskip

{\bf (LR)}
there exists $L$ such that for all  $n\geq 1$,
$l\in \{ 1, \dots , C(n) \}$ and $k \in \{ 1 , \dots , C(n-1) \}$
$$
 h_{l}(n)\leq L \ h_{k}(n-1)\ .
$$
\end{defi}

Most of the basic dynamical properties of linearly
recurrent minimal Cantor systems are described in \cite{CDHM}.
In particular, they are
uniquely ergodic and the unique invariant measure is never strongly mixing.
In addition, $C(n) \leq L$ for any $n \in \NN$ and the set of matrices
$\{M(n); n \geq 1 \}$ is finite.

To prove Theorem \ref{ssi} we will need to consider property
\medskip

{\bf (KR5')} for all $n\geq 2$, $1\leq k \leq C (n-1)$, $1\leq l\leq C
(n)$,
there exist $0 \leq j < j' < h_{l} (n)$ such that $T^{-j} B_{l}(n)
\subseteq B_{k}(n-1)$ and $T^{-j'} B_{l}(n)
\subseteq B_{k}(n-1)$,
\medskip

instead of {\bf (KR5)}.
This condition is equivalent to say that the coefficients of $M(n)$ are
strictly larger than $1$ for $n \geq 2$.

Let $(X,T)$ be a linearly recurrent system
given by a nested sequence of CKR partitions
$(\P(n);n\in\NN)$ which verifies {\bf (KR1)}-{\bf (KR6)} and {\bf (LR)}.
Then the sequence of partitions defined by
$\P'(0)=\P(0)$ and  $\P'(n)=\P(2n-1)$ for $n \geq 1$, is a
sequence of nested CKR partitions of the system
which verifies
{\bf (KR1)}--{\bf (KR4)}, {\bf (KR5')}, {\bf (KR6)} and {\bf (LR)}
(with another constant). It follows that $M'(1)=M(1)$ and
$M'(n)=M(2n-1)M(2n-2)$ for $n \geq 2$, where $(M(n); n\geq 1)$ and
$(M'(n);n \geq 1)$ are the sequence of matrices associated to the partitions
$(\P(n);n \in \NN)$ and $(\P'(n);n \in \NN)$ respectively.
Moreover,
\begin{equation}
\label{KR}
\displaystyle \sum_{n\ge 2} \vvert \alpha P(n) H(1) \vvert^p <\infty \
\Leftrightarrow \
\displaystyle \sum_{n\ge 2} \vvert \alpha P'(n) H(1) \vvert^p
<\infty
\end{equation}
where $\alpha \in \RR$ and $p\in \{1,2\}$.

\section{Markov chain associated to a linearly recurrent system}
\label{markov}

Let $(X,T)$ be a linearly recurrent system and let $\mu$ be its
unique invariant measure. Consider a sequence $(\P (n) ; n \geq 0
)$ of CKR partitions which satisfies {\bf (KR1)}-{\bf (KR6)} and
{\bf (LR)} with constant $L$ and let $(M(n) ; n\geq 1 )$ be the
sequence of matrices associated. The purpose of this section is to
formalize the fact that there exists a Markovian measurable
structure behind the tower structure.

The following relation will be of constant use in the paper. 
For $n\geq 1$ put $\mu(n)=(\mu(B_t(n));1 \leq t \leq C(n))$ (the vector of
measures of the roofs at level $n$). It follows directly from the structure 
of towers that for $1\leq k < n$
\begin{equation}
\label{rel-mesure}
\mu(n-k)=M^T(n-k+1) \cdots M^T(n)\mu(n) \ .
\end{equation}

\subsection{First entrance times and combinatorial structure of the
towers}
In this subsection we define several concepts that will be extensively used
later. An illustration of them is given in Figure \ref{Figure2}.

Define the first entrance time map to the roof $B(n)$, $r_n: X \to \NN$, by
$$r_n(x)= \min \{ j \geq 0 ; T^j(x) \in B(n) \} \ . $$
Since $(X,T)$ is minimal and $B(n)$ is a clopen set, then $r_n$ is
finite and continuous. Define the tower of level $n$ map $\tau_n:
X \to \NN$ by
$$\tau_n(x)=k \text{ if and only if } x \in \T_k(n)
\text{ for some } 1 \leq k \leq C(n)\ . $$
Remark that

\begin{equation}
\label{eq:retour}
r_n(T(x))-r_n(x)=
\begin{cases}
 -1 & \text{ if $x \notin B(n)$}, \\
h_k(n)-1 & \text{ if $x\in B(n)$ and $\tau_n(T(x))=k$}.
\end{cases}
\end{equation}

Let $n \geq 1$ and $1 \leq t \leq C(n)$. By hypothesis {\bf
(KR5)}, several stages in the tower $\T_t(n)$ are included in the
roof $B(n-1)$, in particular stage $B_t(n)$. The number of such
stages is
$$
m_t(n)=\sum_{k =1}^{C(n-1)} m_{t,k}(n)=
\#\{ 0 \leq j < h_t(n) ; T^{-j}B_t(n) \subseteq B(n-1) \} \ .
$$
Let $\{e_1,e_2,\dots,e_{m_t(n)} \}= \{ 0 \leq j < h_t(n) ;
T^{-j}B_t(n) \subseteq B(n-1) \}$ with $h_t (n) > e_1 > e_2 > ...
> e_{m_t(n)}=0$. The integers $e_1,...,e_{m_t(n)}$ are the first entrance
times of points belonging to $\T_t(n)\cap B(n-1)$ into $B_t(n)$.
Moreover, for all $1\leq l \leq m_t (n)$ there is a unique $k \in
\{1,..., C(n-1) \}$ such that
$$
T^{-e_l}B_t(n) \subseteq B_{k}(n-1).
$$
Denote this $k$ by $\theta_l^t(n-1)$. From {\bf (KR6)} we have
\begin{align}
\label{rel:thetakr6}
\theta^t_{m_t(n)} (n-1) = 1 .
\end{align}
We set
\begin{align}
\label{codage}
\theta^t(n-1)=\theta_1^t(n-1) \cdots \theta_{m_t(n)}^t(n-1) \in
\{1,..., C(n-1) \}^* .
\end{align}
Remark that $e_l-e_{l+1}$ is the height of the
$\theta_{l+1}^t(n-1)$-th tower of $\P(n-1)$ for $1\leq l < m_t
(n)$. Thus,
$$
e_l = \sum_{k=l+1}
^{m_t(n)} h_{\theta_k^t(n-1)}(n-1).
$$

Now, the tower $\T_t(n)$ can be decomposed as a disjoint union of
the towers of $\P(n-1)$ it intersects. More precisely,
$\T_t(n)=\bigcup_{l=1}^{m_t(n)} \E_{l,t}(n-1)$, where
$$
\E_{l,t}(n-1)= \bigcup_{j=e_{l-1}-1}^{e_{l}} T^{-j} B_t(n)=
\bigcup_{j=0}^{h_{\theta_l^t(n-1)}(n-1)-1} T^{-j -e_l} B_t(n).
$$
By definition,
$$\E_{l,t}(n-1) \subseteq \bigcup_{j=0}^{h_{\theta_l^t(n-1)}(n-1)-1}
T^{-j} B_{\theta_l^t(n-1)}(n-1).$$

For $x \in X$ denote by $l_n(x)$ the unique integer in
$\{1,...,m_{\tau_n(x)}(n)\}$ such that $x \in
\E_{l_n(x),\tau_n(x)}(n-1)$. The following lemma follows from the
construction. The proof is left to the reader.

\begin{lemma}
\label{important}
For all $x\in X$ we have

\begin{align}
\label{eq:atom}
& \bigcap_{k=1}^n \E_{l_k(x),\tau_k(x)}(k-1)= T^{-r_n(x)}
B_{\tau_n(x)}(n); \\
\label{eq:properties}
&\{x\}= \bigcap_{n \geq 1} \E_{l_n(x),\tau_n(x)}(n-1).
\end{align}
Moreover, given
$$
(t_n; n\geq 0) \in \prod_{n\geq 0} \{1,...,C(n)\}, \ \
(j_n; n\geq 1) \in \prod_{n\geq 1} \{1,...,m_{t_n}(n)\}
$$
such that
$\theta^{t_n}_{j_n}(n-1)=t_{n-1}$ for $n\geq 1$, then there exists
a unique $x \in X$ such that
\begin{align}
\label{eq:union}
& \left( (l_n(x),\tau_n(x)); n\geq 1 \right)=
\left((j_n,t_n); n\geq 1 \right) \ .
\end{align}
\end{lemma}
Remark that the set in \eqref{eq:atom} is the atom of the partition 
$P(n)$ containing $x$. 
\vskip 1cm

For all $n\geq 1$ and $x\in X$ define
$s_{n-1} (x) = (s_{n-1,t} (x) ; 1\leq  t \leq C(n-1) )$ by
$$
s_{n-1,t}(x) = \# \{j;  r_{n-1}(x) < j  \leq r_{n} (x), \  T^j x \in B_{t}(n-1) \} .
$$
It also holds that,
$$s_{n-1,t} (x)= \# \{j; l_{n}(x) < j \leq m_{\tau_{n}(x)}(n) , \
\theta_j^{\tau_{n}(x)}(n-1)=t \}.$$ 
In other words, the vector $s_{n-1}(x)$
counts, in each coordinate $1 \leq t \leq C(n-1)$, the number of
times the tower $\T_t(n-1)$ is crossed by a point $x$, after its
first return to the roof of level $n-1$, and before reaching the
roof of the tower of level $n$ it belongs to. Notice that $s_{n-1}$
does not consider the order in which the towers are visited. 
In the following figure we illustrate the notations introduced
previously.

\begin{figure}[h]
\resizebox*{0.48\textwidth}{!}{\input{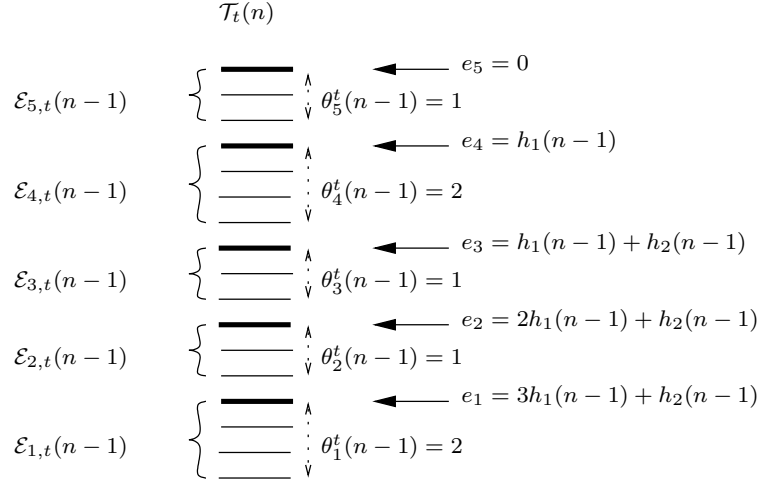}}
\vspace{0.375cm} 
\caption{In the figure we present tower $t$ of $\P(n)$ in a particular 
example. We assume that in  $\P(n-1)$ there are only two towers and 
that $m_t(n)=5$. 
If $x \in \E_{1,t}(n-1)$ then $s_{n-1}(x)=(3,1)^T$ and $l_n(x)=1$. If 
$x \in \E_{4,t}(n-1)$ then $s_{n-1}(x)=(1,0)^T$ and $l_n(x)=4$.
\label{Figure2}}
\end{figure}

A direct computation yields to the following lemma. 
It will be extensively used in the sequel.
Denote by $<\cdot,\cdot>$ the usual scalar product.

\begin{lemma}
\label{le:rn} For all $x\in X$ and all $n\geq 2$ it holds,
\begin{align*}
&r_1(x)=s_0(x); \ r_n(x)=r_{n-1}(x)+ <s_{n-1}(x), H(n-1)> ; \\
&r_n (x) = \sum_{j=2}^{n-1} <s_j (x), P(j) H(1)>+ <s_1(x), H(1)>  + s_0 (x).
\end{align*}
\end{lemma}

\subsection{Markov property for the towers}
Now we prove the sequence of random variables $(\tau_n; n\in \NN)$
is a non-stationary Markov chain. We need some preliminary
computations. Let $n\geq 1$. From Lemma \ref{important} we have
\begin{align*}
\label{eq:propertiesmu}
\mu(B_{\tau_n(x)}(n))
=
\mu \left(\bigcap_{k=1}^n \E_{l_k(x),\tau_k(x)}(k-1)\right).
\end{align*}

Let $(t_i \in \{1,...,C(i)\}; 0 \leq i \leq n)$. 
The set $[\tau_n=t_n]$ is the tower $\T_{t_n}(n)$. For $0 \leq k < n$, 
$\tau_k(x)$ is constant on each level of $\T_{t_n}(n)$. By a simple induction, 
the number of levels of this tower where 
$\tau_0(x)=t_0,\dots,\tau_{n-1}(x)=t_{n-1}$ is equal to 
$m_{t_1,t_0}(1) \cdots m_{t_n,t_{n-1}}(n)$. In other words, the set 
$[\tau_0=t_0,...,\tau_n=t_n]$ is the union of 
$m_{t_1,t_0}(1) \cdots m_{t_n,t_{n-1}}(n)$ levels of the tower $\T_{t_n}(n)$
and 
\begin{equation}\label{eq:taus}
\mu[\tau_0=t_0,...,\tau_n=t_n]=
m_{t_1,t_0}(1) \cdots m_{t_n,t_{n-1}}(n) \mu(B_{t_n}(n))  \ .
\end{equation}

In particular, from the last equality and the definition of the matrices
$(M(n);n \geq 1)$ we deduce
$$
\mu[\tau_n=t_n | \tau_{n-1}=t_{n-1}]= \frac{m_{t_n,t_{n-1}}(n)
\mu(B_{t_n}(n))}
{\mu(B_{t_{n-1}}(n-1))}.
$$
Now, given the sequence $(\P (n) ; n\in \NN)$ we can prove
$(\tau_n;n\in \NN)$ is
a Markov chain on the probability space $(X,\B_X,\mu)$.
Therefore, by \eqref{rel-mesure}, the matrix
$Q(n) = (q_{t,\bar t} (n) ; 1\leq \bar t \leq C(n) ,
1\leq t \leq C(n-1) )$ with
$$
q_{t,\bar t}(n)=
\frac{ m_{\bar t,t}(n) \mu(B_{\bar t}(n)) }{\mu(B_{t}(n-1))}
$$
is a stochastic matrix.

\begin{lemma}
The sequence of random variables $(\tau_n;n\in \NN)$ is a non-stationary
Markov chain with associated stochastic matrices  $(Q(n) ; n \geq 1)$.
\end{lemma}
\begin{proof}
From \eqref{eq:taus} we get
\begin{align*}
&\mu[\tau_{n} = \bar t | \tau_{n-1}=t,\tau_{n-2}=t_{n-2}, \ldots
,\tau_0=t_0 ] \\
&=\frac{ m_{t_1,t_0}(1) \cdots m_{t,t_{n-2}}(n-1) m_{\bar t,t}(n)
\mu(B_{\bar t}(n))}
        {m_{t_1,t_0}(1) \cdots m_{t,t_{n-2}}(n-1) \mu(B_{t}(n-1))}\\
&=\frac{m_{\bar t,t}(n)  \mu(B_{\bar t}(n))}
{\mu(B_{t}(n-1))} \\
&=\mu[\tau_{n}=\bar t | \tau_{n-1}=t]\\
&= q_{t,\bar t}(n).
\end{align*}
\end{proof}
The following lemma provides an exponential mixing property for
non-stationary ergodic Markov chains. It is a standard result. The
proof can be adapted from that of Corollary 2 page 141 of
\cite{Se}. That is, this corollary can be generalized 
to the case of a non-stationary Markov chain where the stochastic 
matrices have not necessarily the same dimension. Alternatively, a direct
proof follows from inequality (3.3) Theorem 3.1 page 81 of \cite{Se} in the 
case of our particular matrices.
 
\begin{lemma}
\label{MixExpo}
Let $(\tau_n;n\in \NN)$ be the non-stationary
Markov chain defined in the previous subsection. There
exist $c\in \RR_+$ and $\beta \in [0,1[ $ such that for all $n,k\in \NN$,
with $k\leq n$,
$$
\sup_{1\leq t\leq C(n-k) , 1\leq \bar t \leq C(n) }|
\mu[\tau_n=\bar t | \tau_{n-k}=t] - \mu[\tau_n=\bar t] | \leq c \beta^{k} \ .
$$

\end{lemma}

\section{Measurable eigenvalues}

\label{sconditions}

The main purpose of this section is to prove Statement
\eqref{measurable}  of Theorem \ref{ssi} (this
is done in Subsection \ref{mainssi}).
In the first subsection we give a general necessary and sufficient
condition
to be a measurable eigenfunction of a minimal Cantor system.

\subsection{A necessary and sufficient condition to be an eigenvalue}

We give a general necessary and sufficient condition to be an
eigenvalue. We do not use it directly to prove our result, but we
think it gives an idea of the classical way to tackle the problem
and shows that the difficulty relies in understanding the
stochastic behavior of the sequence $(r_n;n \in \NN)$. We would
like to stress the fact that we still do not have a convincing
interpretation of the sequence of functions $\rho_n$ which appears
in the next theorem.
\begin{theo}
\label{vpns}
Let $(X,T)$ be a minimal Cantor system and let $\mu$ be an invariant
measure. Let $(\P(n) ; n\in\NN )$ be a sequence of CKR partitions verifying
{\bf (KR1)}-{\bf (KR4)}.
A complex number $\lambda = \exp (2i\pi \alpha )$ is an eigenvalue of
$(X,T)$ with respect to $\mu$ if and only if
there exist real functions
$\rho_n : \{1,...,C(n)\} \rightarrow \RR$, $n\in \NN$, such that
\begin{equation}
\label{vp}
\alpha
\left( r_n (x) +  \rho_n \circ \tau_n (x)  \right)
\hbox{ converges } ( {\rm mod} \  \ZZ )
\end{equation}
for $\mu$-almost every $ x\in X$ when  $n$ tends to infinity.
\end{theo}
\begin{proof}
Let $\lambda=\exp (2i\pi \alpha )$ be a complex number of modulus
1 such that (\ref{vp}) holds and let $g$ be the corresponding limit function.
Consider $x\not\in \displaystyle\cap_{n\in \NN} B(n)$, so
$x$ does not belong to $B(n)$ for all large enough $n \in \NN$.
Then, from \eqref{eq:retour} we get
\begin{equation*}
\frac{\exp (2i\pi g(Tx))}{\exp (2i\pi g (x))}
=
\lim_{n\rightarrow \infty} \lambda^{r_n (Tx) - r_n (x) }
=
\lambda^{-1}.
\end{equation*}
This implies $\lambda$ is an eigenvalue of $(X,T)$ with respect to
$\mu$.
\medskip

Now, assume $\lambda$ is an eigenvalue of $(X,T)$ with respect to
$\mu$ and let $g\in L^2(X,\B_X,\mu)$ be an associated
eigenfunction. For all $n\in \NN$ let $\phi_n =
\lambda^{-r_n}$ and  $\psi_n  = g  / \phi_n$.
The map
$\phi_n $ is $\P(n)$-measurable and bounded, then
$$
\phi_n \EE_\mu(\psi_n | \P(n) ) = \EE_\mu(\phi_n \psi_n | \P(n))=
\EE_\mu (g | \P(n) ) \xrightarrow[n\to \infty]{} g
$$
$\mu$-almost everywhere. Since
$\psi_n \circ T^{-j} / \psi_n = \lambda^{r_n\circ T^{-j}-r_n-j}$,
the restriction of $\psi_n$ to each tower of level $n$ is invariant under $T$. Thus
$\EE_\mu(\psi_n | \P(n))$ is constant on each of these towers and is
therefore equal to the average of $\psi_n$ on
each tower.

To finish, for $1 \leq i \leq C(n)$ we define $\rho_n (i)$ such that
$$\hbox{Arg}\lambda^{-\rho_n (i)} =
\hbox{Arg} \left(\frac{1}{\mu ( B_{i}(n))} \int_{B_{i}(n)}
\psi_n \d \mu \right).$$
This ends the proof.
\end{proof}

\begin{remark}
The same proof works if we remove the Cantor and clopen hypotheses.
\end{remark}

\subsection{Eigenvalues of linearly recurrent systems}
\label{mainssi}
In this subsection we prove Statement \eqref{measurable} of
Theorem \ref{ssi}.
Recall $(X,T)$ is linearly recurrent and $\mu$ is the
unique invariant measure.  Let
$(\P (n) ; n \geq 0 )$ be a sequence of CKR partitions such that
{\bf (KR1)}-{\bf (KR6)} and {\bf (LR)} with constant $L$ are satisfied.
Let $(M(n) ; n\geq 1 )$ be the associated
sequence of matrices.

We will need the following lemma. 
Its proof can be found in \cite{CDHM}.
\begin{lemma}
\label{tendto0} Let $u \in \RR^{C(1)}$ be a real vector such that
$\vvert P(n)u\vvert\to 0$ as $n\to\infty$. Then, there exist $m \geq 2
$, an integer vector $w \in \ZZ^{C(m)}$ and a real vector $v
\in \RR^{C(m)}$ with
$$
P(m)u=v + w\text{ and } \Vert P(n,m)v\Vert\to 0\text{ as
}n\to\infty\ .
$$
\end{lemma}


%

Assume the following condition holds:
\begin{equation}
\label{conditioncarre}
\displaystyle \sum_{n\ge 2} \vvert \alpha P(n)  H(1) \vvert^2
<\infty \ .
\end{equation}
Then, $\vvert P(n)(\alpha H(1)) \vvert \to 0$ as
$n \to \infty$.
From Lemma \ref{tendto0} there exist an integer $n_0\geq 2$, a real vector
$v\in \RR^{C(n_0)}$ and an
integer vector $w \in \ZZ^{C(n_0)}$
such that, $P(n_0)(\alpha H(1)) = v+ w $ and
$P(n,n_0) v \to 0$ as $n \to \infty$.
By modifying a finite number of towers, if needed, we can assume
without loss of generality that $n_0=1$ and that $H(1)=(1,...,1)^T$.
So condition \eqref{conditioncarre} implies
\begin{equation}
\displaystyle \sum_{n\ge 2} \Vert P(n) v \Vert^2
<\infty \ .
\end{equation}
\medskip
From \eqref{KR}, we can also assume without loss of generality
that {\bf (KR5')} holds. 
That is, entries of matrices $M(n)$ are larger than
$2$ for all $n\geq 2$.

For $n\geq 1$ we define $g_n : X \rightarrow \RR$ by
$$
g_n(x) = s_0(x) + <s_1(x), v>+ \sum_{j=2}^{n-1}  <s_j(x), P(j) v>  \  .
$$
Since we are assuming $H(1)=(1,...,1)^T$, then $s_0=0$ and
$$g_n(x) = \sum_{j=1}^{n-1}  <s_j(x), P(j) v> \ ,$$
where we set $P(1)=Id$.
\begin{lemma}
\label{convergence}
If (\ref{conditioncarre}) holds, then the sequence
$(f_n = g_n  - \EE_\mu (g_n) ; n\geq 1)$ converges in $L^2(X,\B_X,\mu)$.
\end{lemma}
\begin{proof}
Let $n \geq 1$.  Recall that $\P(n)$ is the partition of level $n$ and
let $\T(n)$ be the coarser partition $\{\T_j(n); 1 \leq j \leq C(n) \}$.
As usual we identify the 
finite partitions with the $\sigma$-algebras they span and
we use the same notation. Thus $\T(n)$ is the $\sigma$-algebra spanned
by the random variable $\tau_n$.

Let $X_n$ be the random variable given by
$$X_n=<s_n, P(n) v> - \EE_\mu (<s_n, P(n) v>)\ .$$ 
We decompose it as $X_n=Y_n+Z_n$ where
$$Y_n=\EE_\mu (X_n|\P(n)) \text{ and }
Z_n=<s_n, P(n) v>-\EE_\mu (<s_n, P(n) v>|\P(n))\ . $$

We write $\kappa_n=\Vert P(n) v \Vert$. 
Observe that for some positive constant $K$ and all
$n\geq 1$ we have 
$|X_n| \leq K \kappa_n$, $|Y_n| \leq K \kappa_n$
and $|Z_n| \leq K \kappa_n$. 

First we show that the series $\sum Z_n$ converges. Let $m$ and 
$n$ be positive integers with $m < n$. The random variable $Z_m$ is measurable 
with respect to $\P(m+1)$, thus also with respect to $\P(n)$. Since 
$\EE_\mu(Z_n|\P(n))=0$ we get $\EE_\mu(Z_m \cdot Z_n)=0$. 
As $|Z_n|\leq K \kappa_n$ for every $n \geq 1$, the series
$\sum \EE_\mu(Z_n^2)$ converges, and thus the orthogonal series $\sum Z_n$ 
converges in $L^2(X,\B_X,\mu)$.

Now we prove that the series $\sum Y_n$ 
converges in $L^2(X,\B_X,\mu)$.  Fix $j\geq 1$ and $1\leq \bar t \le C(n+1)$.
The set $\E_{j,\bar t}(n)$ is included in the tower $\T_t(n)$ where 
$t=\theta_j^{\bar t}(n)$. Moreover, the intersection of all levels 
of $\T_t(n)$ with $\E_{j,\bar t}(n)$ are levels of $\T_{\bar t}(n+1)$ 
(see Figure \ref{Figure2}) and thus have the same measure 
$\mu(B_{\bar t}(n+1))$. As each level of the tower $\T_t(n)$ has measure 
$\mu (B_t(n))$ we have

\begin{equation}\label{mesure-preuve}
\mu(\E_{j,\bar t}(n)|\P(n))(x)=
\left\{
\begin{array}{ll}
\frac{\mu(B_{\bar t}(n+1))}{\mu(B_t(n))} & \text{ if } x \in \T_t(n) \\
0 & \text{ otherwise. }
\end{array} \right.
\end{equation}

Observe that this conditional probability is constant in each 
atom of $\T(n)$ and thus 
$$\mu(\E_{j,\bar t}(n)|\P(n))=\mu(\E_{j,\bar t}(n)|\T(n)) \ .$$
As $s_n$ is constant on each set $\E_{j,\bar t}(n)$, the same 
property holds for $X_n$ and thus 

\begin{equation}\label{mesure-preuve2}
Y_n=\EE_\mu(X_n | \P(n))=\EE_\mu(X_n | \T(n)) \ .
\end{equation}

In particular $Y_n$ is equal to a constant on each set 
$[\tau_n=\bar t \ ]$ and we write $y_{\bar t}$ for this constant.
Fix $k$ with $0\leq k \leq n$. If $\tau_{n-k}(x)=t$ we have
$$
\EE_\mu(Y_n|\T(n-k))(x)=\sum_{\bar t=1}^{C(n)} 
\mu[\tau_n=\bar t \ | \tau_{n-k}=t] y_{\bar t} \ .
$$
We have 
$$
\sum_{\bar t=1}^{C(n)}\mu[\tau_n=\bar t\ ]y_{\bar t}=\EE_\mu(Y_n)=0 \ .
$$
We deduce from
Lemma \eqref{MixExpo}, the 
fact that $C(n)$ is bounded independently of $n$ and 
$|Y_n| \leq  K \kappa_n$ that for some positive constant $C$ 

$$|\EE_\mu(Y_n|\T(n-k))(x)|\leq 
\sum_{\bar t=1}^{C(n)} |\mu[\tau_n=\bar t| \tau_{n-k}=t]-\mu[\tau_n=\bar t]|
y_{\bar t} \leq C \beta^k \kappa_n \ .
$$
As $Y_{n-k}$ is measurable with respect to $\T(n-k)$ we have

\begin{equation*}
|\EE_\mu(Y_n\cdot Y_{n-k})| \leq C \beta^k \kappa_n \EE_\mu(|Y_{n-k}|) 
\leq C \beta^k \kappa_n \kappa_{n-k} \ .
\end{equation*}

For $1\leq m < n$ we compute 
\begin{eqnarray*}
\EE_\mu\left (\left (\sum_{k=m}^n Y_k \right )^2 \right ) 
&=& \sum_{m\leq j,l\leq n} \EE_\mu
(Y_j \cdot Y_l)\leq C \sum_{m\leq j,l\leq n} \beta^{j-l}\kappa_j\kappa_l \\
&=& C \sum_{r=0}^{n-m} \beta^r \sum_{l=m}^{n-r} \kappa_l \kappa_{l+r} 
\leq C \sum_{r=0}^{n-m} \beta^r \sum_{l=m}^{n} \kappa_l^2 \\
&\leq&  \frac{C}{1-\beta} \sum_{l=m}^{n} \kappa_l^2 \ .
\end{eqnarray*}

Since the series $\sum \kappa^2_j$ converges, the partial sums of the series 
$\sum Y_n$ form a Cauchy sequence in $L^2(X,\B_X,\mu)$.
\end{proof} 


The following lemma completes the proof of 
Theorem \ref{ssi} \eqref{measurable}.

\begin{lemma}
Let $f\in L^2(X,\B_X,\mu)$ be the limit of sequence $(f_n ; n\geq 1)$. The
function $\exp ( 2i \pi f )$
is an eigenfunction of $(X,T)$ with respect to $\mu$
associated to the eigenvalue $ \exp (2i\pi \alpha )$.
\end{lemma}
\begin{proof}
Remark that $g_n (x) = \alpha r_{n-1} (x) \ (\text{\rm mod} \ \ZZ)$. From relation
 \eqref{eq:retour} we get
$$
f_n (Tx) = f_n (x) -\alpha  \ (\text{\rm mod}
\ \ZZ)
$$
holds outside of the roof $B(n)$ and $\mu(B(n))\to 0$ as $n\to \infty$.
We conclude using Lemma \ref{convergence}.
\end{proof}

\section{Continuous Eigenvalues of Linearly Recurrent Systems}

Let $(X,T)$ be a linearly recurrent dynamical system with constant $L$. The
main purpose of this section is to prove the necessary condition
in the statement \eqref{continuous} of Theorem \ref{ssi}.
We recall the sufficient condition was proven in \cite{CDHM}.

\subsection{A necessary and sufficient condition to be a continuous eigenvalue}

In this subsection we only assume that
$(\P(n); n\in\NN )$ is a sequence of CKR partitions describing the system
$(X,T)$ which satisfies {\bf (KR1)}-{\bf (KR6)}. We give a general
necessary and sufficient condition to be a continuous eigenvalue.

\begin{prop}
\label{condcont}
Let $\lambda=\exp(2i\pi\alpha)$
be a complex number of modulus 1. The following conditions
are equivalent,

\begin{enumerate}
\item
$\lambda$ is a continuous eigenvalue of the minimal Cantor system $(X, T)$;
\item
$(\lambda^{r_n (x)} ; n\geq 1)$ converges uniformly in $x$, i.e., the
sequence $(\alpha r_n(x) ; n\geq 1)$ converges
$({\rm mod} \ \ZZ )$ uniformly in $x$.
\end{enumerate}
\end{prop}

\begin{proof}
We start proving that (1) implies (2). Let $g$ be a continuous
eigenfunction associated to  $\lambda$. For all $n\geq 1$ and all
$x\in X$ we have $T^{r_n (x)} (x) \in B(n)\subseteq B_1 (n-1)$
(the last inclusion is due to {\bf (KR6)}). Hence, using {\bf
(KR3)}, we deduce that $\lim_{n\rightarrow \infty}  T^{r_n (x)}
(x) = u$ uniformly in $x$, where $u$ is the unique element of
$\cap_{n\geq 0} B(n)$. The eigenfunction $g$ being uniformly
continuous we have that $\lambda^{r_n (x)} = g(T^{r_n (x)}
(x))/g(x)$ tends to $g(u)/g(x)$ uniformly in $x$.

Now we prove (2) implies (1). We set $\phi (x) =
\lim_{n\rightarrow \infty}  \lambda^{r_n (x)}$. Since the
convergence is uniform and $r_n$ is continuous, then $\phi$ is
continuous.

Let $x$ be such that $x\not \in B(n)$ for infinitely many $n$.
Then, from \eqref{eq:retour}, we obtain $\phi(T(x))=\lambda^{-1}
\phi(x)$. Using the minimality of $(X,T)$ and the continuity of
$\phi$, we obtain that $\phi (T (y)) = \lambda^{-1} \phi (y)$ for
all $y\in X$. Consequently $\lambda$ is a continuous eigenvalue.
\end{proof}

\begin{coro}
\label{contgen}
Let $\lambda$ be a complex number of modulus $1$.
\begin{enumerate}
\item
If $\lambda$ is a continuous eigenvalue of $(X,T)$ then
$$
\lim_{n\to \infty} \lambda^{h_{j_n} (n)} = 1
$$
uniformly in $(j_n; n\in \NN) \in \prod_{n\in \NN} \{1,\dots,C(n)\}$.
\item
If
$$
\sum_{m\ge 1}
\left(
\frac{\sup_{k\in \{1,...,C(m+1)\}} h_{k}(m+1)}{\inf_{k\in \{1,...,C(m)\}}
h_{k}(m)}
\right)
\sup_{k\in \{1,...,C(m)\}}
\mid\lambda^{h_{k}(m)} -1\mid <\infty
$$
then $\lambda$ is a continuous eigenvalue of $(X,T)$.
\end{enumerate}
\end{coro}

\begin{proof}
Let $g$ be a continuous eigenfunction of $\lambda$. Then, it is uniformly continuous. Let $\epsilon > 0$.
There exists $n_0  \in \NN$
such that $|g (y) -g (u)| < \epsilon / 2$ for all $y\in B_1 (n_0)$, where $\{ u \} = \cap_{n\in \NN} B(n)$.

Let  $(j_n ; n\in \NN ) \in \prod_{n\in \NN} \{1,\dots,C(n)\}$.
For all $n\in \NN$ we take $x(n) \in B_{j_n} (n)$ and we set $y(n)
= T^{-h_{j_n} (n)} (x(n)) \in B(n)$. Hence, using {\bf (KR6)}, for
all $n\geq n_0+1$ the points  $x(n)$ and $y(n)$ belong to $B(n)
\subseteq B_1 (n_0)$. Consequently
\begin{align*} |
\lambda^{h_{j_n} (n)} - 1| & =
| g (T^{h_{j_n} (n)} y(n))   - g(y(n))| \\
& \leq
| g (x(n)) - g(u)| + |g(u) - g(y(n))|
< \epsilon .
\end{align*}

Now we prove (2). It suffices to remark, by Lemma \ref{le:rn},
that for all $x\in X$ and all $0< n < m$,
$$
\mid \lambda^{r_m (x)} - \lambda^{r_n (x)} \mid
 =
\mid 1 - \lambda^{r_m (x) - r_n (x)} \mid
$$
$$
\leq
\sum_{l=n}^{m-1} \frac{\sup_{k\in \{1,...,C(l+1)\}}
h_{k}(l+1)}{\inf_{k\in \{1,...,C(l)\}} h_{k}(l)}
\sup_{k\in \{1,...,C(l)\}} \mid 1-\lambda^{h_{k}(l)}\mid .
$$

Hence, from Proposition \ref{condcont}, $\lambda$ is a continuous eigenvalue.
\end{proof}
Remark that for linearly recurrent systems statement (2) gives the 
sufficient condition for $\lambda$ to be a continuous eigenvalue. This was proved 
in \cite{CDHM}.

\subsection{The linearly recurrent case}
Now we assume $(X,T)$ is linearly recurrent and we prove Theorem
\ref{ssi} part (2). We also assume without loss of generality that
the sequence of partitions verifies {\bf (KR5')}, that is entries
of $M(n)$ are bigger than $2$ for any $n\geq 2$ (see discussion in
subsection \ref{linrec}). To prove the result we introduce an
intermediate statement which gives a more precise interpretation
to the necessary condition.

\begin{prop}
\label{th:condcont}
Let $\lambda=\exp(2i\pi\alpha) $
be a complex number of modulus 1.
The following properties are equivalent.
\begin{enumerate}
\item
$\lambda$ is a continuous eigenvalue of the minimal Cantor system $(X,T)$.

\item There exist $n_0 \in \NN$, $v \in \RR^{C(n_0)}$,
$z \in \ZZ^{C(n_0)}$, such that $\alpha P(n_0) H(1)=v + z$,
$P(n,n_0) v \to 0$ as $n\to \infty$ and the series
$$
\sum_{j \geq n_0+1} <s_j(x), P(j,n_0) v >
$$
converges for every $x \in X$.
\item
$\sum_{n\geq 2} \vvert \alpha  P(n) H(1) \vvert < \infty$
\end{enumerate}
\end{prop}

\begin{proof}

In \cite{CDHM} it is proven that (3) implies (1).

\medskip

We prove that (1) implies (2): assume $\lambda$ is a continuous
eigenvalue of $(X,T)$. We deduce from statement (1) of Corollary
\ref{contgen} that $\vvert \alpha P(n) H(1) \vvert$ converges to
$0$ as $n$ tends to $\infty$. By Lemma \ref{tendto0}, there are
$n_0 \in \NN$, $v \in \RR^{C(n_0)}$ and  $z \in \ZZ^{C(n_0)}$ such
that $\alpha P(n_0) H(1)= v + z$ and $P(n,n_0) v \to 0$ as $n\to
\infty$. By modifying a finite number of towers we can assume
without loss of generality that $n_0=1$.

By Lemma \ref{le:rn}, for $n\geq 1$ and $x \in X$, $r_n(x) =
\sum_{j=1}^{n-1} <s_j(x), P(j) H(1)>  + s_0 (x)$, where we put
$P(1)=I$. Then, 
$$\alpha r_n(x)= \sum_{j=1}^{n-1} <s_j(x), P(j)
v> + \sum_{j=1}^{n-1} <s_j(x), P(j) z> + \alpha s_0(x).$$ 
From
Proposition \ref{condcont}, $\sum_{j=1}^{n-1} <s_j(x), P(j) v>
+ \alpha s_0(x) \to v(x) \ (\text{\rm mod}\  \ZZ)$ as $n \to
\infty$. We distinguish two cases: if $v(x) \in (0,1)$, we write
$\sum_{j=1}^{n-1} <s_j(x), P(j) v> + \alpha s_0(x) = V_n(x) +
v_n(x)$ with $V_n(x) \in \ZZ$ and $v_n(x) \in [0,1)$; if $v(x)=0$
we consider $v_n(x) \in [-1/2,1/2)$. Then, in both cases, $(v_n
(x) ; n\in \NN)$ converges and a fortiori $(v_{n+1}(x)-v_n(x); n
\geq 1) \to 0$ as $n \to \infty$. Moreover,

\begin{align*}
& \sum_{j=1}^{n} <s_j(x), P(j) v> - \sum_{j=1}^{n-1} <s_j(x), 
P(j) v> \\ & =
<s_n(x), P(n) v> = V_{n+1}(x)-V_n(x)+v_{n+1}(x) - v_n(x).
\end{align*}

Since, for a linearly recurrent system $\{ s_n (x) ; x\in X , n\in
\NN \}$ is bounded, $P(n)v \to 0$ and $(v_{n+1}(x)-v_n(x))\to 0$
as $n \to \infty$. We conclude $V_n(x)$ is a constant integer for
all large enough $n \in \NN$. Consequently the series 
$\sum_{j\geq 2} <s_j(x),
P(j) v>$ converges.

\medskip
Now we prove that (2) implies (3):
we assume, without loss of generality,  that $n_0=1$ and that for any 
$x \in X$ the series
$$
\sum_{j\geq 2} <s_j(x), P(j) v> \in \RR 
$$
converges.
It suffices to prove that $\sum_{j\geq 2} \Vert P(j) v \Vert <
\infty$.

For $n \geq 2$ define $i(n) \in \{1,...,C(n) \}$
such that
$$|<e_{i(n)}, P(n) v> |=\max_{i \in \{1,...,C(n)\}} |<e_i, P(n) v>|$$
where $e_{i}$ is the $i$-th canonical vector of
$\RR^{C(n)}$. Let

$$
I^+=\{n \geq 2; <e_{i(n)}, P(n) v> \geq 0 \}, \
I^-=\{n \geq 2 ; <e_{i(n)}, P(n) v> <0\}.
$$

To prove
$\sum_{j\geq 2} \Vert P(j) v \Vert < \infty$
we only need to show

$$
\sum_{j\in I^+} <e_{i(j)}, P(j) v> < \infty \hbox{ and } -
\sum_{j\in I^-} <e_{i(j)}, P(j) v> < \infty.
$$

Since arguments we will use are similar in both cases we only
prove the first fact. To prove $\sum_{j\in I^+} <e_{i(j)}, 
P(j) v> < \infty$ we only show
\begin{equation}
\label{eq:iplus} \sum_{j\in I^+\cap 2\NN} <e_{i(j)}, P(j) v> <
\infty,
\end{equation}
and analogously it can be proven

$$
\sum_{j\in I^+\cap (2\NN+1)} <e_{i(j)}, P(j) v> < \infty.
$$

We construct two points $x,y \in X$ such that
$s_n(x)-s_n(y)=e_{i(n)}$ if
$n \in I^+\cap 2\NN$  and $s_n(x)-s_n(y)=0$ elsewhere.
By hypothesis, from this fact we conclude \eqref{eq:iplus}.

To construct $x$ and $y$, according to Lemma \ref{important}, we
only need to produce sequences

$$
(t_n; n \in \NN) \in \Pi_{n \in \NN} \{1,...,C(n)\}, \ \ (j_n;
n\geq 1) \in \Pi_{n \geq 1} \{1,...,m_{t_n}(n)\}
$$
and
$$
(\bar t_n; n \in \NN) \in \Pi_{n \in \NN} \{1,...,C(n)\}, \ \
(\bar j_n; n\geq 1) \in \Pi_{n \geq 1} \{1,...,m_{\bar t_n}(n)\}
$$
such that

\begin{align}
\label{rel:thetat}
\theta_{j_n}^{t_n}(n-1)=t_{n-1}
\hbox{ and }
\theta_{\bar j_n}^{\bar t_n}(n-1)=\bar t_{n-1} \hbox{ for all } n\geq 1.
\end{align}
The point
$x$ is the unique one such that $\tau_n(x)=t_n$ and  $l_n(x)=j_n$. Point
$y$ is defined analogously with respect to $\bar t_n$ and $\bar j_n$.
Given $n \in (I^+ \cap 2\NN)^c$ put $t_n=\bar t_n= 1$.

For $n \in I^+ \cap 2\NN$, by property {\bf (KR5')}, there exist
$k \in \{1,...,m_1(n+1)-1\}$ such that $\theta_{k+1}^1(n)=i(n)$.
Put

$$
\bar t_n =  i(n), \ \bar j_n = m_{\bar t_n} (n), \
t_n=\theta_k^1(n) \hbox{ and } j_n = m_{ t_n} (n).
$$

Using \eqref{rel:thetakr6} we obtain $\theta_{j_n}^{t_n} (n-1) = 1$ and $\theta_{\bar j_n}^{\bar t_n} (n-1)= 1$.
Then, we set $\bar t_{n+1} = t_{n+1} = 1$, $\bar j_{n+1} = k+1$ and $j_{n+1} = k$. Consequently, the relations \eqref{rel:thetat} are satisfied for $n$ and $n+1$.

Now we treat the remaining case: $n \in (I^+)^c \cap 2\NN$. We recall that $t_n = \bar t_n = t_{n+1} = \bar t_{n+1} = 1$. It suffices to set

$$
j_n
=
\bar j_n
=
m_{1} (n)
\hbox{ and }
j_{n+1}
=
\bar j_{n+1}
=
m_{1} (n+1)
$$
to fulfill the relations \eqref{rel:thetat}.

For each $n\in I^+ \cap 2\NN$, the towers of level $n$ visited by
$x$ and $y$ after their first entrance time to $B(n)$ and before
their first entrance time to $B(n+1)$ are
$$
\S_n (x) = \{\theta^{1}_{k+1} (n), \dots ,\theta^{1}_{m_{1}
(n+1)} (n)\} \hbox{ and } \S_n (y) =\{ \theta^{1}_{k+2} (n),
\dots ,\theta^{1}_{m_{1} (n+1)} (n)\}
$$
respectively. Therefore, $s_n(x)-s_n(y)=e_{i(n)}$.

On the other hand, if $n\not \in I^+ \cap 2\NN$ then $\S_n (x)$
and $\S_n (y)$ are the empty set. Hence $s_n (x)=s_n(y)=0$.
\end{proof}

\section{Example: measurable and non continuous eigenvalues}
\label{example}

We construct explicitly a system with a nontrivial Kronecker
factor but having a trivial equicontinuous factor. Let us consider
the commuting matrices
$$
A=\left[ \begin{array}{cc}
5 & 2\\
2 & 3
\end{array}\right]
\hbox{ and } B=\left[
\begin{array}{cc}
2 & 1\\
1 & 1
\end{array}
\right].
$$
We set $\varphi = \frac{1+\sqrt{5}}{2}$. Let $e= (\varphi,
1)^T$, $f = (-1, \varphi)^T$, $\alpha_A = 3+2\varphi$, $\beta_A
= 5-2\varphi$, $\alpha_B = 1 + \varphi$ and $\beta_B = 2 -
\varphi$. Observe that $\alpha_A > \alpha_B> \beta_A>1>\beta_B>0$
and $\{e, f \}$ is a base of $\RR^2$ made of the common
eigenvectors associated to eigenvalues $\alpha_A, \beta_A$ of $A$
and  $\alpha_B, \beta_B$ of $B$ respectively.

We define recursively the sequence $(v_n;n \geq 1)$ of real
numbers by: $v_1 = 1$ and for all $n > 1 $
$$v_{n+1} =
\left\{
\begin{array}{ll}
\beta_A v_n & \hbox{ if } n v_n \leq 1\\
\beta_B v_n & \hbox{ if } n v_n > 1
\end{array}
\right.
$$
Notice that the sequence $(nv_n;n \geq 1)$ is uniformly bounded
and uniformly bounded away from $0$. Now let $H(1)=M(1)=(1,1)^T$
and for $n \geq 1$
$$M(n+1) =
\left\{
\begin{array}{ll}
A  & \hbox{ if } n v_n \leq 1  \\
B  & \hbox{ if } n v_n > 1
\end{array}
\right.
$$
Remark $M(n)=A$ for infinitely many values of $n$.

Define the words in $\{1,2\}^*$, $\theta^1(A)=2211111$,
$\theta^2(A)=22211$, $\theta^1(B)=211$ and 
$\theta^2(B)=21$. Let$(X,T )$ be a minimal Cantor system such that
there is a sequence of CKR partitions $(\P(n);n\in \NN)$ verifying
{\bf (KR1)}-{\bf (KR6)} with associated sequence of matrices
$(M(n); n \geq 1)$. Moreover, we require that
for $n \geq 1$ and $t\in \{1,2\}$, 
$\theta^t(n)=\theta^t(M(n+1))$ holds (see \eqref{codage} 
for the definition 
of $\theta^t(n)$). This is possible by
\cite{HPS} using Bratteli diagrams. It is clear that $(X,T)$ is
linearly recurrent. We call $\mu$ its unique ergodic measure.

A symbolic way to see this system is by considering the
substitutions $\sigma_A: \{1,2\} \to \{1,2\}^*$, $\sigma_A(1)=
2211111$, $\sigma_A(2)=22211$, and $\sigma_B: \{1,2\} \to
\{1,2\}^*$, $\sigma_B(1)=211$, $\sigma_B(2)=21$. Define a sequence
of substitutions $(\sigma_n;n \geq 1)$ by $\sigma_1 = Id$ and, for
all $n>1$, $\sigma_{n+1} =  \sigma_{n} \circ \sigma_{M(n)}$. It
follows that $... 1 1 1 \sigma_n(1).\sigma_n(2) 2 2 2 ...$
converges to some $\omega \in \{1,2\}^\ZZ$, where the dot
indicates the position to the left of $0$ coordinate. We set $X =
\overline{\{ T^n (\omega), n\in \ZZ\}}$, where $T$ is the shift
map.

Before to study the system $(X,T)$ defined by this sequence of matrices
we need a general property. We keep notations of previous sections.
\begin{lemma}
\label{v-ortho} Let $v \in \RR^{C(1)}$. If $\lim_{n\to \infty}
\Vert P(n) v \Vert= 0$, then $v$ is orthogonal to the vector
$\mu(1)=(\mu (B_k (1)) ; 1\leq k \leq C(1))^T$.
\end{lemma}
\begin{proof}
Let $v \in \RR^{C(1)}$ be such that $\lim_{n\to \infty} \Vert P(n)
v \Vert= 0$. Then, for $n>1$
\begin{align*}
\vert <\mu(1), v> \vert &= \vert < P^T(n)\mu(n), v> \vert \\
&= \vert <\mu(n), P(n)v> \vert \\
&\leq \Vert P(n)v \Vert ,
\end{align*}

and the last term converges to $0$ as $n\to\infty$. Thus $v$ is
orthogonal to $\mu(1)$.
\end{proof}

\begin{prop}
\label{example}
Let $(X,T)$ be the linearly recurrent system defined above.
The set of eigenvalues of $(X,T)$ is

$$
E_\mu=\left\{ \exp{(2i\pi \alpha)} \in \CC ; \alpha = \left(
\varphi - 1, 2-\varphi \right) A^{-l} w, \ l \geq 0, w\in \ZZ^2
\right\}
$$
None of these eigenvalues is continuous except the trivial one.
\end{prop}
\begin{proof}
Let $v = - (\varphi-2) f = (\varphi - 2 , \varphi -1)^T$ and $n
\geq 2$. Hence, $P(n)v=P(n)(\varphi H(1))$ (mod $\ZZ^2)$. Also,
since $v$ is an eigenvector of $A$ and $B$, from the definition of
$v_n$ we get
$$
P(n) v = \beta_{M(n)}\cdots \beta_{M(2)}v = v_n v.
$$
The sequence $(v_n)_{n \geq 1}$ was constructed so that $n v_n $
is uniformly bounded and uniformly bounded away from $0$. It
follows that

\begin{equation}\label{sumv}
\sum_{n\geq 2} v_n = \infty \text{ and } \sum_{n\geq 2} v_n^2 < \infty \ ,
\end{equation}
and
\begin{equation}
\label{somme-finie}
\sum_{n\geq 2} \Vert P(n) v \Vert =  \infty \hbox{ and }
\sum_{n\geq 2} \Vert P(n) v \Vert^2  <  \infty.
\end{equation}
In particular, $\lim_{n\to \infty} \Vert P(n) v \Vert = 0$ and, by
Lemma \ref{v-ortho}, $v$ is orthogonal to $\mu(1)=(\mu (B_1(1)),
\mu (B_2 (1))^T$.
\medskip

{\bf Claim:} {\it Let $\alpha \in \RR$ and $\lambda=\exp{(2i\pi \alpha)}$:  
$\vvert P(n)(\alpha H(1))
\vvert \to 0$ as $n\to \infty$ holds if and only if $\lambda \in
E_\mu$. Moreover, if $\lambda \in E_\mu$ then $\vvert P(n)(\alpha
H(1))\vvert = c \Vert P(n)v\Vert$, for some positive constant $c$.}
\medskip

{\it Proof of the claim.} First assume $\vvert P(n)(\alpha H(1))
\vvert \to 0$ as $n\to \infty$ holds. By Lemma \ref{tendto0},
there exist $m \geq 2$, an integer vector $w \in \ZZ^{C(m)}$ and a
real vector $v'\in \RR^{C(m)}$ with $P(m) (\alpha H(1))=v' + w$ and
$\Vert P(n)P(m)^{-1}v'\Vert\to 0$ as $n\to\infty$. From Lemma
\ref{v-ortho}, vector $P(m)^{-1}v'$ is orthogonal to $\mu(1)$.
Hence, there exists $k\in \RR$ such that $P(m)^{-1}v'=kv$ and
\begin{equation}
\label{delta} P(m)( \alpha H(1)) = k P(m) v + w.
\end{equation}
Suppose $k=0$. It is not difficult to show by induction that ${\rm
gcd} (h_1(m),h_2(m))=1$. Then, since $w$ is an integer vector,
$\alpha \in \ZZ$ and $\lambda=1$ which belongs to
$E_\mu$.

Suppose $k\not = 0$. Then, $k = W_1 - W_2$ where $P(m)^{-1} w=
(W_1 , W_2)^T$. This gives,
\begin{align}
\label{varphi} \alpha H(1) & = \left(
\begin{matrix} \varphi - 1 & 2-\varphi \\ \varphi - 1  & 2-\varphi
\end{matrix}
\right)
P(m)^{-1} w .
\end{align}
The determinants of the matrices $A$ and $B$ are respectively
equal to 11 and 1. Therefore, since $P(m)=A^{l_m}B^{k_m}$ for some
$l_m,k_m \geq 0$,
$$\label{varphi} \alpha =(\varphi - 1, 2-\varphi) A^{-l_m} w',$$
with $w'\in \ZZ^2$. So $\lambda \in E_\mu$.

Conversely, let $\lambda \in E_\mu$. Then, since $M(n)=A$ for
infinitely many $n\geq 2$, for $n$ large enough we get
$$P(n)(\alpha H(1))=P(n) 
\left[
\begin{matrix}
\varphi-2 &-(\varphi -2) \\
 \varphi-1 &-(\varphi-1) \\
\end{matrix} \right]
+ w,
$$
where $w \in \ZZ^2$. Therefore, $\vvert P(n)(\alpha H(1))\vvert = c
\Vert P(n)v\Vert$, for some positive constant $c$, which proves
the claim.
\medskip

Finally, the proposition follows from Theorem \ref{ssi} and
property \eqref{somme-finie}.
\end{proof}

This proposition gives an example of a minimal dynamical system with a
nontrivial Kronecker factor and a trivial maximal equicontinuous
factor.

\bigskip

{\bf Acknowledgments.} 
We would like to thank Claude Dellacherie for his important advice
about the references concerning convergence of ``almost''
Martingale processes and Fran\c cois Parreau for illuminating
discussions. 

The first author thanks the CMM-CNRS 
who made this collaboration possible. 
The final version of this paper was written while the second
and third authors visited the 
Max-Planck Institute of 
Mathematics (Bonn). The support and hospitality of 
both institutions are very much appreciated.

The third author
acknowledge financial support from Programa Iniciativa
Cient\'{\i}fica Milenio P01-005 and FONDECYT 1010447. This project
was also partially supported by the international cooperation program 
ECOS-Conicyt C03-E03. 

We also thank the referee for many valuable comments.

\end{document}